\documentclass[10pt, article]{amsart}

\usepackage{lmodern} 
\usepackage[T1]{fontenc}
\usepackage[utf8]{inputenc}
\usepackage{amsmath, amssymb, amsfonts, amscd, verbatim}
\usepackage{tabularx}
\usepackage[normalem]{ulem}
\usepackage{hyperref}
\usepackage{indentfirst}
\usepackage{latexsym}
\usepackage{faktor}
\usepackage{xfrac} 
\usepackage{xcolor}
\usepackage{bm}
\usepackage{tikz}

\theoremstyle{plain}
\newtheorem{Proposition}{Proposition}[section]
\newtheorem{Theorem}[Proposition]{Theorem}
\newtheorem{Corollary}[Proposition]{Corollary}
\newtheorem{Lemma}[Proposition]{Lemma}
\theoremstyle{definition}
\newtheorem{Definition}[Proposition]{Definition}
\newtheorem{Remark}[Proposition]{Remark}
\newtheorem{Example}[Proposition]{Example}

\numberwithin{equation}{section}
\title[An Algebraic Viewpoint on Linear Differential Equations]
{An Algebraic Viewpoint on Linear Differential Equations}

\author{Hussain Al-Rasheed}
\address{Jubail Industrial College, Al Jubail, Eastern Province, Saudi Arabia}
\email{rashedhs@rcjy.edu.sa}

\date{\today}

\begin{document}

\begin{abstract}
Classical methods for solving linear ordinary differential equations, such as superposition, the method of undetermined coefficients, and the annihilator technique, are often presented as heuristic, procedural rules. In this article, we show that these methods admit a coherent algebraic interpretation when constant-coefficient linear differential operators are viewed as elements of the polynomial ring $\mathbb{C}[D]$, acting on spaces of smooth functions. 

Without invoking the formalism of $D$-modules or non-commutative operator algebras, we explain how homogeneous solution spaces arise as kernels of linear operators, how particular solutions form affine cosets, and how the search for solutions is an infinite-dimensional eigenvalue problem. Furthermore, we extend this algebraic framework to variable-coefficient equations, resolving the Euler equation through ring isomorphisms and framing d'Alembert's reduction of order as non-commutative operator factorization. We also explore the boundary of this linear theory, demonstrating how diffeomorphic linearization allows certain non-linear equations---such as those of Bernoulli and Riccati---to be mapped directly into the $\mathbb{C}[D]$-module framework. Finally, we contrast this framework with the multivariable ring $\mathbb{C}[D_x, D_y]$, using the loss of the principal ideal domain property to explain the intrinsic structural divergence of partial differential equations, and indicate further universal extensions to discrete difference equations and the Weyl algebra.
\end{abstract}

\maketitle

\tableofcontents

\section{Introduction}

Linear ordinary differential equations with constant coefficients occupy a standard place in the undergraduate mathematics, engineering, and applied science curricula \cite{BoyceDiPrima}. Despite their foundational importance, the standard presentation of solution techniques often emphasizes procedural mechanics over structural understanding \cite{Rota1998}. Students are instructed to ``guess'' particular solutions, to multiply by the independent variable in cases of resonance, and to memorize substitutions for variable-coefficient equations, often without an explanation of why these steps are mathematically necessary.

The objective of this article is to present an algebraic viewpoint that renders these classical methods transparent. The central observation is that a constant-coefficient linear differential operator
\[
L = a_n D^n + a_{n-1} D^{n-1} + \cdots + a_0,
\qquad D = \frac{d}{dx},
\]
may be regarded as an element of the polynomial ring $\mathbb{C}[D]$, acting linearly on suitable spaces of smooth functions \cite{BirkhoffRota}. From this perspective, solving the differential equation $L[y] = f$ amounts to solving a linear equation in an infinite-dimensional vector space. The homogeneous solutions form the kernel of the operator, and the general solution forms an affine translate of that kernel. Furthermore, solving this equation is an infinite-dimensional eigenvalue problem, where finding the kernel equates to finding the eigenfunctions of the operator associated with the zero eigenvalue.

\section{The Algebraic Setup}

To establish this framework, we define the foundational objects of study. The central structure is the polynomial ring $\mathbb{C}[D]$, where $D = \frac{d}{dx}$. Elements of $\mathbb{C}[D]$ are polynomial expressions in $D$ with complex coefficients and are interpreted as linear differential operators. Because the complex numbers $\mathbb{C}$ form a field, $\mathbb{C}[D]$ is a \emph{principal ideal domain} (PID). This critical structural property guarantees that every ideal is generated by a single polynomial, ensuring that concepts such as the greatest common divisor and coprime factorizations are well-defined for differential operators \cite{BirkhoffRota}.

We consider the space $C^\infty(\mathbb{R})$, the set of all infinitely differentiable complex-valued functions on $\mathbb{R}$. This space acts naturally as a $\mathbb{C}[D]$-module via the operation
\[
p(D)\cdot f := p\!\left(\tfrac{d}{dx}\right)f.
\]
In this context, linear differential operators are module endomorphisms. 

For any $f\in C^\infty(\mathbb{R})$, the annihilator $\operatorname{Ann}_{\mathbb{C}[D]}(f)$ consists of all constant-coefficient operators that annihilate $f$. Thus, finding the solution space to a homogeneous differential equation $p(D)y = 0$ is the process of identifying the torsion elements of the module that are annihilated by $p(D)$.

\section{Homogeneous Equations and the Zero-Eigenspace}

Let $p(D)\in\mathbb{C}[D]$ be a monic polynomial (without loss of generality, as any linear operator may be divided by its non-zero leading coefficient without altering its kernel). Consider the linear differential equation
\[
p(D)y=0.
\]
Since $p(D)$ acts as a $\mathbb{C}$-linear endomorphism of $C^\infty(\mathbb{R})$, the space of all solutions is the kernel of this operator:
\[
\mathcal{S}(p)=\ker\bigl(p(D):C^\infty(\mathbb{R})\to C^\infty(\mathbb{R})\bigr).
\]
From an operator-theoretic standpoint, this kernel is the \emph{eigenspace of the operator $p(D)$ associated with the eigenvalue zero}. Therefore, solving the homogeneous equation is equivalent to identifying an eigenfunction basis for this zero-eigenspace. 

To understand how this basis is constructed, we recall that an \emph{eigenfunction} of a differential operator $L$ is a non-zero smooth function $f(x)$ satisfying $L(f(x)) = \lambda f(x)$, or equivalently, $(L - \lambda I)f = 0$. For the fundamental derivative operator $D = \frac{d}{dx}$, we determine its standard eigenfunctions by solving the eigenvalue equation:
\[
(D - \lambda)f(x) = 0.
\]
This is a separable first-order differential equation and, up to multiplication by a scalar constant, the unique solution $f(x)$ is the exponential function $e^{\lambda x}$. 

Because polynomial multiplication is commutative, any polynomial operator $p(D) \in \mathbb{C}[D]$ preserves these eigenfunctions. Applying $p(D)$ to $e^{\lambda x}$ yields:
\[
p(D)e^{\lambda x} = p(\lambda)e^{\lambda x}.
\]
Consequently, the exponential $e^{\lambda x}$ belongs to the zero-eigenspace $\ker p(D)$ if and only if $p(\lambda) = 0$. However, when the algebraic roots of the characteristic polynomial are repeated, the operator is defective and cannot be diagonalized into independent pure exponential eigenfunctions. To complete the basis, we are forced to construct \emph{generalized eigenfunctions}. 

A generalized eigenfunction of rank $m$ corresponding to the eigenvalue $\lambda$ is a non-zero function $f(x)$ that is annihilated by the $m$-th power of the shifted operator, meaning it satisfies $(D - \lambda)^m f = 0$, while $(D - \lambda)^{m-1} f \neq 0$. We accomplish this explicit basis construction in three progressive steps.

\subsection*{Step 1: The Base Operator $p(D) = D^m$}
We begin by examining the foundational case where $\lambda = 0$ is a root of multiplicity $m$. The generalized eigenfunctions must satisfy $D^m f(x) = 0$.

\begin{Lemma}
For $m\in\mathbb{N}$, the set $\{1,x,x^2,\dots,x^{m-1}\}$ constitutes a basis for $\ker D^m$, yielding:
\[
\ker D^m = \operatorname{span}_{\mathbb{C}}\{1,x,x^2,\dots,x^{m-1}\}.
\]
\end{Lemma}

\begin{proof}
A function $f\in C^\infty(\mathbb{R})$ satisfies $D^m f=0$ if and only if its $m$-th derivative vanishes. By repeated integration, such a function must be a polynomial of degree at most $m-1$. Conversely, any polynomial of degree at most $m-1$ is annihilated by $D^m$.
\end{proof}

\subsection*{Step 2: The Shifted Operator $p(D) = (D-\lambda)^m$}
To extend our basis to an arbitrary root $\lambda$ of multiplicity $m$, we introduce an automorphism. For $\lambda\in\mathbb{C}$, define conjugation by exponentials as:
\[
T_\lambda : C^\infty(\mathbb{R}) \to C^\infty(\mathbb{R}), \qquad
(T_\lambda f)(x)=e^{\lambda x}f(x).
\]
This map is a $\mathbb{C}$-linear automorphism of $C^\infty(\mathbb{R})$, with inverse $T_{-\lambda}$. A direct computation establishes that $T_{-\lambda}\circ (D-\lambda)\circ T_\lambda = D$.

\begin{Lemma}
Let $m\in\mathbb{N}$ and $\lambda\in\mathbb{C}$. A function $f\in C^\infty(\mathbb{R})$ satisfies
\[
f\in \ker (D-\lambda)^m
\]
if and only if
\[
T_{-\lambda}f \in \ker D^m.
\]
\end{Lemma}

\begin{proof}
Using the identity $(D-\lambda)^m = T_\lambda \circ D^m \circ T_{-\lambda}$, we observe that
\[
(D-\lambda)^m f=0
\quad \Longleftrightarrow \quad
D^m(T_{-\lambda}f)=0.
\]
The result follows immediately.
\end{proof}

\begin{Corollary}
For $m\in\mathbb{N}$ and $\lambda\in\mathbb{C}$, the set $\{e^{\lambda x}, xe^{\lambda x}, \dots, x^{m-1}e^{\lambda x}\}$ constitutes a basis for $\ker (D-\lambda)^m$, yielding:
\[
\ker (D-\lambda)^m
=
\operatorname{span}_{\mathbb{C}}
\{e^{\lambda x}, xe^{\lambda x}, \dots, x^{m-1}e^{\lambda x}\}.
\]
\end{Corollary}

This corollary provides the operator-theoretic formulation for resonance. The polynomial multiples $x^k e^{\lambda x}$ serve as the infinite-dimensional analogues of generalized eigenvectors completing a Jordan chain \cite{BirkhoffRota}.

\subsection*{Step 3: The General Polynomial Operator $p(D)$}
Finally, we assemble the complete basis for an arbitrary operator $p(D)$ by exploiting the algebraic structure of the PID $\mathbb{C}[D]$.

\begin{Theorem}
Let $\lambda_1 \neq \lambda_2$ and $m_1,m_2\in\mathbb{N}$. Then, as vector spaces,
\[
\ker\!\bigl((D-\lambda_1)^{m_1}(D-\lambda_2)^{m_2}\bigr)
=
\ker (D-\lambda_1)^{m_1}
\;\oplus\;
\ker (D-\lambda_2)^{m_2}.
\]
\end{Theorem}

\begin{proof}
Since $(D-\lambda_1)^{m_1}$ and $(D-\lambda_2)^{m_2}$ are coprime elements of the PID $\mathbb{C}[D]$, there exist polynomials $a(D),b(D)\in\mathbb{C}[D]$ such that
\[
a(D)(D-\lambda_1)^{m_1}+b(D)(D-\lambda_2)^{m_2}=1.
\]
Applying this identity to any $f\in \ker\bigl((D-\lambda_1)^{m_1}(D-\lambda_2)^{m_2}\bigr)$ yields a direct sum decomposition of $f$ into components lying in the respective kernels. The intersection of these kernels is trivial, as distinct exponential-polynomial types are linearly independent.
\end{proof}

\begin{Corollary}
Let
\[
p(D)=(D-\lambda_1)^{m_1}(D-\lambda_2)^{m_2}\cdots(D-\lambda_r)^{m_r},
\]
where $\lambda_1,\dots,\lambda_r\in\mathbb{C}$ are pairwise distinct roots. Then the total zero-eigenspace is the direct sum of the generalized eigenspaces:
\[
\ker p(D)
=
\bigoplus_{j=1}^r \ker (D-\lambda_j)^{m_j}.
\]
In particular, the set $\{x^k e^{\lambda_j x} : 0\le k\le m_j-1,\; 1\le j\le r\}$ constitutes an explicit basis for $\ker p(D)$, yielding:
\[
\ker p(D) = \operatorname{span}_{\mathbb{C}}\{x^k e^{\lambda_j x} : 0\le k\le m_j-1,\; 1\le j\le r\}.
\]
\end{Corollary}

\begin{Remark}
An immediate consequence of this explicit basis construction is that for any non-zero linear differential operator $p(D) \in \mathbb{C}[D]$, the dimension of its zero-eigenspace is equal to its polynomial degree:
\[
\dim \ker p(D) = \deg p(D).
\]
\end{Remark}

\subsection{The Torsion Submodule}
The construction of these generalized eigenspaces reveals a fundamental algebraic classification of the function space itself. In abstract algebra, an element $m$ of a module over a ring is designated a \emph{torsion element} if there exists a non-zero element $r$ in the ring such that $rm = 0$. 

Because every solution to a homogeneous constant-coefficient linear differential equation $p(D)y = 0$ is annihilated by the polynomial operator $p(D)$, every solution is a torsion element of the module $C^\infty(\mathbb{R})$. 

We define the torsion submodule, $M_{\text{tors}}$, as the set of all such functions:
\[
M_{\text{tors}} = \left\{ f \in C^\infty(\mathbb{R}) \mid \exists\, p(D) \in \mathbb{C}[D] \setminus \{0\} \text{ such that } p(D)f = 0 \right\}.
\]
By Corollary 3.5, we have identified the composition of this submodule. The torsion submodule of $C^\infty(\mathbb{R})$ is the $\mathbb{C}$-vector space possessing the basis of all functions of exponential-polynomial type (frequently termed \emph{quasipolynomials}):
\[
M_{\text{tors}} = \operatorname{span}_{\mathbb{C}}\{x^k e^{\lambda x} : k \in \mathbb{Z}_{\ge 0},\; \lambda \in \mathbb{C}\}.
\]
Consequently, the zero-eigenspace $\ker p(D)$ of any polynomial differential operator is a finite-dimensional submodule residing entirely within $M_{\text{tors}}$. Functions that cannot be formed by finite linear combinations of quasipolynomials (such as $\ln(x)$ or $\tan(x)$) are \emph{torsion-free}; they will never be annihilated by any constant-coefficient linear differential operator.

\subsection{Systems of Equations and Matrix Diagonalization}
This operator-theoretic perspective extends to systems of linear differential equations, bridging the continuous calculus of ODEs with the discrete algebra of matrices \cite{HirschSmale}. Consider the first-order system $\mathbf{y}' = A\mathbf{y}$, where $A$ is an $n \times n$ matrix of constant complex coefficients. Our objective is to determine the explicit basis for the zero-eigenspace of the matrix differential operator $(D - A)$. 

To construct this basis, we demonstrate the canonical isomorphism between the geometric eigenvectors of $A$ and the continuous eigenfunctions of $(D - A)$.

\begin{Lemma}\label{lem:isomorphism}
Let $V_\lambda = \ker(A - \lambda I)$ denote the geometric eigenspace of the matrix $A$ associated with the eigenvalue $\lambda$. A vector-valued function of the form $\mathbf{f}(x) = \mathbf{v}e^{\lambda x}$ (where $\mathbf{v} \neq \mathbf{0}$) lies in $\ker(D - A)$ if and only if $\mathbf{v} \in V_\lambda$. Moreover, the mapping $\mathbf{v} \mapsto \mathbf{v}e^{\lambda x}$ defines a canonical isomorphism between the discrete geometric eigenspace $V_\lambda$ and the continuous pure-exponential zero-eigenspace of $(D - A)$.
\end{Lemma}

\begin{proof}
Suppose $\mathbf{f}(x) = \mathbf{v}e^{\lambda x} \in \ker(D - A)$. Applying the operator yields:
\[
(D - A)(\mathbf{v}e^{\lambda x}) = \mathbf{v}D(e^{\lambda x}) - A\mathbf{v}e^{\lambda x} = \lambda\mathbf{v}e^{\lambda x} - A\mathbf{v}e^{\lambda x} = (\lambda I - A)\mathbf{v}e^{\lambda x}.
\]
For this expression to equal the zero vector, and given that $e^{\lambda x} \neq 0$ for all $x$, we must have $(\lambda I - A)\mathbf{v} = \mathbf{0}$. This implies $A\mathbf{v} = \lambda\mathbf{v}$, confirming that $\mathbf{v}$ is an eigenvector of $A$ and $\mathbf{v} \in V_\lambda$. The converse follows identically. The map is linear and bijective, thereby establishing the isomorphism.
\end{proof}

As a direct consequence of this canonical isomorphism, the complete eigenfunction basis for the differential system is derived from the diagonalization of the matrix $A$.

\begin{Theorem}
Let $A$ be an $n \times n$ diagonalizable matrix with linearly independent eigenvectors $\mathbf{v}_1, \dots, \mathbf{v}_n$ corresponding to eigenvalues $\lambda_1, \dots, \lambda_n$. A smooth vector-valued function $\mathbf{f}(x)$ lies in the zero-eigenspace $\ker(D - A)$ if and only if it can be expressed as a linear combination of the pure-exponential functions $\mathbf{v}_j e^{\lambda_j x}$. Consequently, the explicit basis for $\ker(D-A)$ is $\{ \mathbf{v}_1 e^{\lambda_1 x}, \dots, \mathbf{v}_n e^{\lambda_n x} \}$.
\end{Theorem}

\begin{proof}
By Lemma~\ref{lem:isomorphism}, any function of the form $\mathbf{v}_j e^{\lambda_j x}$ belongs to the kernel. Conversely, assume $(D - A)\mathbf{f}(x) = \mathbf{0}$, which implies $\mathbf{f}'(x) = A\mathbf{f}(x)$.
Because $A$ is diagonalizable, it admits the decomposition $A = P\Lambda P^{-1}$, where $P$ is the matrix whose columns are the eigenvectors $\mathbf{v}_j$, and $\Lambda$ is the diagonal matrix of eigenvalues $\lambda_j$. Substituting this yields $\mathbf{f}'(x) = P\Lambda P^{-1}\mathbf{f}(x)$.
Defining the transformed variable $\mathbf{g}(x) = P^{-1}\mathbf{f}(x)$, we obtain the algebraically decoupled system $\mathbf{g}'(x) = \Lambda \mathbf{g}(x)$.
Since $\Lambda$ is diagonal, this system uncouples into $n$ independent scalar equations $g_j'(x) = \lambda_j g_j(x)$. By Section 3.1, the solution to each is $g_j(x) = c_j e^{\lambda_j x}$. Applying the transformation matrix $P$ recovers the original function:
\[
\mathbf{f}(x) = P\mathbf{g}(x) = \sum_{j=1}^n c_j e^{\lambda_j x} \mathbf{v}_j.
\]
Therefore, every function in the zero-eigenspace of $(D-A)$ is restricted to a linear combination of the geometric eigenvectors $\mathbf{v}_j$ scaled by $e^{\lambda_j x}$.
\end{proof}

When the matrix $A$ is defective (not diagonalizable), the required generalized eigenvectors pair with polynomial multipliers, mirroring the scalar methodology articulated in Section 3.2. 

\begin{Example}[Defective Matrices and Generalized Eigenvectors]
Consider the system $\mathbf{y}' = A\mathbf{y}$ where $A = \begin{pmatrix} \lambda & 1 \\ 0 & \lambda \end{pmatrix}$. The matrix possesses a repeated eigenvalue $\lambda$ but only a single linearly independent eigenvector $\mathbf{v}_1 = \begin{pmatrix} 1 \\ 0 \end{pmatrix}$. The canonical eigenfunction is $\mathbf{y}_1(x) = \mathbf{v}_1 e^{\lambda x}$. 

To complete the basis, we recall from scalar theory that a repeated root requires the polynomial multiplier $x$. However, a naive trial function $\mathbf{v}_1 x e^{\lambda x}$ fails; applying $(D-A)$ yields a non-zero remainder of $\mathbf{v}_1 e^{\lambda x}$ due to the product rule of differentiation. 

To annihilate this remainder, we must append a constant vector term $\mathbf{v}_2 e^{\lambda x}$, yielding the trial function:
\[
\mathbf{y}_2(x) = \mathbf{v}_1 x e^{\lambda x} + \mathbf{v}_2 e^{\lambda x}.
\]
Applying $(D-A)$ to this function and setting it to $\mathbf{0}$ mandates that $(A - \lambda I)\mathbf{v}_2 = \mathbf{v}_1$. This is the definition of a generalized eigenvector. For our specific matrix, computing this yields $\mathbf{v}_2 = \begin{pmatrix} 0 \\ 1 \end{pmatrix}$. Thus, the polynomial multiplier $x$ developed in scalar theory is preserved, but it algebraically forces the creation of a generalized eigenvector to maintain the kernel.
\end{Example}

Furthermore, we can elevate this algebraic decoupling to handle higher-order systems governed by the generalized operator $(p(D) - A)$. Note that $p(D)$ is a scalar polynomial operator; when applied to a vector function, it acts identically on each component, meaning it is algebraically equivalent to the operator $p(D)I$. Consider the classical coupled oscillator system $\mathbf{y}'' = A\mathbf{y}$, which equates to finding the kernel of the operator $(D^2 - A)$. Applying this operator to the pure-exponential trial function $\mathbf{v}e^{rx}$ yields:
\[
(p(D) - A)\mathbf{v}e^{rx} = p(r)I\mathbf{v}e^{rx} - A\mathbf{v}e^{rx} = (p(r)I - A)\mathbf{v}e^{rx} = \mathbf{0}.
\]
This algebraic reduction proves two simultaneous structural requirements: the constant spatial vector $\mathbf{v}$ must be an eigenvector of the matrix $A$ associated with some eigenvalue $\lambda_j$, and the scalar rate $r$ must be a root of the shifted scalar polynomial $p(r) = \lambda_j$. Solving a coupled system of higher-order differential equations is thereby reduced to discrete matrix diagonalization and sequential scalar polynomial factorization.

\section{Non-Homogeneous Equations and Constant Coefficients}

In this section we study linear non-homogeneous ordinary differential equations of the form
\[
p(D)y=f,
\]
where $p(D)\in\mathbb{C}[D]$ and the forcing term $f$ resides within the torsion submodule $M_{\text{tors}}$. Our objective is to reinterpret classical solution methods using the algebraic framework developed in the previous sections.

The guiding principle is to construct a finite-dimensional $\mathbb{C}$-vector space $F\subset C^\infty(\mathbb{R})$, designated as a \emph{trial space}, satisfying the following algebraic properties:
\begin{enumerate}
\item $F$ is invariant under $D$, and hence under $p(D)$;
\item the restriction $\left.p(D)\right|_F : F\to F$ is an automorphism;
\item $F$ is minimal in dimension, ensuring that no redundant functions are included in the trial basis;
\item $\ker p(D)\cap F=\{0\}$.
\end{enumerate}
Under these conditions, the equation $p(D)y=f$ admits a unique solution in $F$, which serves as a particular solution of the differential equation \cite{BirkhoffRota}.

\begin{Remark}
While the space $F$ isolates a unique particular solution, the selection of the trial space itself is not unique. Distinct admissible trial spaces may generate different particular solutions; however, any two such solutions differ by an element of $\ker p(D)$. This reflects the inherent affine geometry of the general solution space, $\mathcal{S}_f$, visualized in Figure~\ref{fig:affine_coset}.
\end{Remark}

\begin{figure}[htbp]
\centering
\begin{tikzpicture}[scale=1.2]
    \draw[->, thick, gray!80] (-1,0,0) -- (3,0,0) node[right]{};
    \draw[->, thick, gray!80] (0,-1,0) -- (0,3,0) node[above]{$C^\infty(\mathbb{R})$};
    \draw[->, thick, gray!80] (0,0,-1) -- (0,0,3) node[below left]{};

    \filldraw[fill=blue!10, draw=blue!80, thick, opacity=0.7]
        (-1, -0.5, 0) -- (2.5, 1.25, 0) -- (1.5, 2.25, 0) -- (-2, 0.5, 0) -- cycle;
    \node[blue!80!black] at (0.7, 0.6, 0) {$\ker p(D)$};

    \draw[->, thick, red, >=stealth] (0,0,0) -- (0,1.5,1.5) node[midway, left]{$y_p$};

    \filldraw[fill=red!10, draw=red!80, thick, opacity=0.7]
        (-1, 1, 1.5) -- (2.5, 2.75, 1.5) -- (1.5, 3.75, 1.5) -- (-2, 2, 1.5) -- cycle;
    \node[red!80!black] at (1, 2.2, 0.6) {$\mathcal{S}_f = y_p + \ker p(D)$};

    \filldraw[black] (0,0,0) circle (1.5pt) node[anchor=north west] {$\mathbf{0}$};
    \filldraw[red] (0,1.5,1.5) circle (1.5pt);
\end{tikzpicture}
\caption{The algebraic structure of the general solution to $p(D)y = f$. The particular solution $y_p$ acts as a translation vector, shifting the homogeneous subspace $\ker p(D)$ into the parallel affine coset $\mathcal{S}_f$.}
\label{fig:affine_coset}
\end{figure}
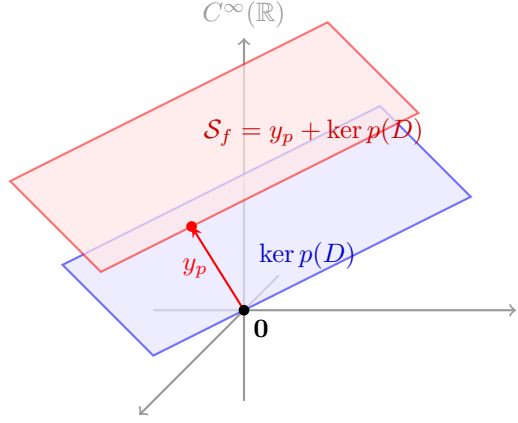

\subsection{The Method of Undetermined Coefficients}

We first formalize the classical method of undetermined coefficients.

Let the forcing function be
\[
f(x)=x^{k_0} e^{\mu x},
\]
where $k_0\in\mathbb{N}\cup\{0\}$ and $\mu\in\mathbb{C}$. Let $m_0 \ge 0$ be the algebraic multiplicity of $\mu$ as a root of $p(D)$ (where $m_0 = 0$ if $p(\mu) \neq 0$).

For any $r \in \mathbb{N}\cup\{0\}$, define the trial space
\[
F_r := x^r \cdot \operatorname{span}_{\mathbb{C}}\{D^n f : n\ge 0\},
\]
and recall the homogeneous subspace associated with the exponential rate $\mu$:
\[
\ker (D-\mu)^{m_0}
=
\operatorname{span}_{\mathbb{C}}\{e^{\mu x}, xe^{\mu x}, \ldots, x^{m_0-1}e^{\mu x}\}.
\]
(If $m_0=0$, this space is trivially $\{0\}$).

\begin{Lemma}
With the notation established above, the following statements hold:
\begin{enumerate}
\item $F_r = x^r e^{\mu x}\,\mathbb{C}[x]_{\le k_0}$;
\item $F_r\cap \ker(D-\mu)^{m_0}=\{0\}$ if and only if $r\ge m_0$.
\end{enumerate}
\end{Lemma}

\begin{proof}
By repeated differentiation, the generated space is
\[
F_r
=
x^r e^{\mu x}\mathbb{C}[x]_{\le k_0}
=
\operatorname{span}_{\mathbb{C}}\{x^r e^{\mu x}, x^{r+1}e^{\mu x}, \ldots, x^{r+k_0}e^{\mu x}\},
\]
so the space $F_r$ has dimension $k_0+1$. On the other hand,
\[
\ker(D-\mu)^{m_0}
=
\operatorname{span}_{\mathbb{C}}\{e^{\mu x}, xe^{\mu x}, \ldots, x^{m_0-1}e^{\mu x}\}.
\]
An intersection exists only if the powers of $x$ overlap. Because the lowest degree in $F_r$ is $r$, and the highest degree in the kernel is $m_0-1$, the spaces are disjoint if and only if $r > m_0-1$. As $r$ and $m_0$ are integers, this implies:
\[
F_r \cap \ker(D-\mu)^{m_0}=\{0\}
\quad \Longleftrightarrow \quad
r\ge m_0.
\]
\end{proof}

\begin{Theorem}
There exists a finite-dimensional vector subspace $F \subset C^\infty(\mathbb{R})$ in which the equation $p(D)y=f$ has a unique solution.
\end{Theorem}

\begin{proof}
We construct $F$ by taking the minimal shift $r_0 = m_0$ dictated by the preceding lemma, defining $F = F_{m_0}$. By construction, $F \cap \ker(D-\mu)^{m_0} = \{0\}$. Because any element in $F$ is a quasipolynomial of the form $P(x)e^{\mu x}$, its intersection with the generalized eigenspaces of $p(D)$ associated with roots other than $\mu$ is trivially zero. By the direct sum decomposition of $\ker p(D)$ established in Corollary 3.5, this guarantees that $F \cap \ker p(D) = \{0\}$.

Because $F$ is a finite-dimensional trial space, invariant under $p(D)$, and satisfies $F \cap \ker p(D) = \{0\}$, the restriction $\left.p(D)\right|_F : F \to F$ is an injective endomorphism. In finite dimensions, injectivity mathematically guarantees bijectivity. Thus, $\left.p(D)\right|_F$ is an automorphism, and $p(D)y=f$ possesses a unique solution residing in $F$.
\end{proof}

The generalization of the theorem proceeds in two steps.

\noindent\textbf{Step 1: Superposition of Forcing Terms}\\
Let $p(D)$ be an arbitrary operator in $\mathbb{C}[D]$, and let the forcing function be a finite sum of exponential-polynomials where the rates $\lambda_j \in \mathbb{C}$ are pairwise distinct (grouping any terms with identical rates by their maximum degree $k_j$):
\[
f(x) = \sum_{j=1}^N x^{k_j} e^{\lambda_j x}.
\]

\begin{Theorem}
With $p(D)$ and $f(x)$ defined as above, there exists a finite-dimensional vector subspace $F$ in which $p(D)y = f$ has a unique solution.
\end{Theorem}

\begin{proof}
By the preceding theorem, for each individual forcing term $f_j(x) = x^{k_j} e^{\lambda_j x}$, there exists a finite-dimensional trial space $F_j$ such that the restricted operator $\left.p(D)\right|_{F_j}$ is an automorphism. Consequently, each sub-equation $p(D)y = f_j$ has a unique solution $y_j \in F_j$. Because the exponential rates $\lambda_j$ are pairwise distinct, the trial spaces intersect trivially. Thus, by the linearity of $p(D)$, the direct sum $F = \bigoplus_{j=1}^N F_j$ is a finite-dimensional, $p(D)$-invariant trial space that contains the superposition solution $y_p = \sum_{j=1}^N y_j$.
\end{proof}

\noindent\textbf{Step 2: The Fully Factored Operator}\\
We now complete the generalization by evaluating the required trial space when $p(D)$ is the product of linear coprime factors raised to their maximum powers.

Let
\[
p(D)=(D-\mu_1)^{m_1}(D-\mu_2)^{m_2}\cdots(D-\mu_d)^{m_d},
\]
where $\mu_1,\ldots,\mu_d\in\mathbb{C}$ are pairwise distinct, and let
\[
f(x) = \sum_{j=1}^N x^{k_j} e^{\lambda_j x},
\]
where the exponential rates $\lambda_1, \ldots, \lambda_N$ are pairwise distinct.

\begin{Theorem}
For the operator $p(D)$ and forcing function $f(x)$ defined above, the minimal finite-dimensional trial space $F$ in which $p(D)y=f$ has a unique solution is given by the direct sum
\[
F = \bigoplus_{j=1}^N F_{r_j}, \qquad \text{where } F_{r_j} = x^{r_j} \cdot \operatorname{span}_{\mathbb{C}}\{D^n (x^{k_j} e^{\lambda_j x}) : n \ge 0\},
\]
and the explicitly required algebraic enlargement for each individual term is
\[
r_j = m(\lambda_j),
\]
where $m(\lambda_j)$ is the algebraic multiplicity of $\lambda_j$ as a root of $p(D)$ (specifically, $m(\lambda_j) = m_i$ if $\lambda_j = \mu_i$, and $0$ otherwise).
\end{Theorem}

\subsection{The Annihilator Method}

We now reinterpret the method of undetermined coefficients utilizing the language of quotient spaces.

Let $p(D)y=f$, where $p(D)\in\mathbb{C}[D]$ and $f\in C^\infty(\mathbb{R})$ is a function of exponential-polynomial type.

\begin{Definition}[Annihilator]
Let $f \in C^\infty(\mathbb{R})$. A nonzero polynomial $q(D)\in\mathbb{C}[D]$ is designated an \emph{annihilator} of $f$ if $q(D)f=0$. An annihilator $q(D)$ is of \emph{minimal degree} if for any polynomial $\tilde q(D)\in\mathbb{C}[D]$ satisfying $\tilde q(D)f=0$, one possesses $\deg q(D)\le \deg \tilde q(D)$.
\end{Definition}

Identifying a minimal annihilator $q(D)$ and applying it to both sides of the non-homogeneous equation yields the homogeneous equation:
\[
p(D)q(D)y=0.
\]

Consequently, any particular solution of $p(D)y=f$ must reside within the space $\ker\bigl(p(D)q(D)\bigr)$. By the linearity of the operator, the difference between any two particular solutions is annihilated by $p(D)$, meaning they differ by an element of $\ker p(D)$. This isolates the quotient space:
\[
\faktor{\ker\bigl(p(D)q(D)\bigr)}{\ker\bigl(p(D)\bigr)},
\]
which acts to parametrize particular solutions modulo the homogeneous solutions.

\begin{Proposition}\label{prop:annihilator-dimension}
Let $p(D),q(D)\in\mathbb{C}[D]$ with $q(D)\neq 0$. Then
\[
\dim\left(\faktor{\ker\bigl(p(D)q(D)\bigr)}{\ker\bigl(p(D)\bigr)}\right)
=
\deg q(D).
\]
\end{Proposition}

\begin{proof}
By the dimension property established in Section 3.3, the dimension of the kernel of any constant-coefficient operator equals its algebraic degree. Thus, $\dim \ker\bigl(p(D)q(D)\bigr) = \deg\bigl(p(D)q(D)\bigr) = \deg p(D) + \deg q(D)$. The dimension of the quotient space then follows from the standard linear algebra identity $\dim(V/W)=\dim V-\dim W$, yielding $\deg q(D)$.
\end{proof}

\begin{Theorem}[Algebraic Equivalence and Dimension]\label{thm:equivalence}
Let $f \in M_{\text{tors}}$ have a minimal annihilator $q(D)$. Let $F$ be a minimal trial space for $p(D)y=f$ satisfying the four guiding principles of this section, and let
\[
V_q = \faktor{\ker\bigl(p(D)q(D)\bigr)}{\ker\bigl(p(D)\bigr)}
\]
be the quotient space defined by the Annihilator Method. Then the natural quotient projection restricted to $F$, 
\[
\left.\pi\right|_F : F \to V_q,
\]
is a canonical isomorphism. Consequently, $\dim F = \deg q(D)$, proving the two methods are mathematically equivalent.
\end{Theorem}

\begin{proof}
Because $F$ is $D$-invariant and contains $f$, it contains the sequence \[f,\,Df,\,D^2f,\dots,D^{k-1}f,\] where $k=\deg q(D)$. Minimality of $q(D)$ implies these specific functions are linearly independent; otherwise, a polynomial of smaller degree would annihilate $f$. Hence, $\dim F\ge \deg q(D)$.

Furthermore, because $F$ is generated by $f$ and its derivatives, applying the minimal annihilator to the entire space yields $q(D)F = \{0\}$. This implies that $F \subset \ker\bigl(q(D)\bigr)$, and consequently, $F \subset \ker\bigl(p(D)q(D)\bigr)$.

Consider the natural quotient projection map $\pi : \ker\bigl(p(D)q(D)\bigr) \to V_q$. By the foundational definition of the trial space (Condition 4), $F \cap \ker\bigl(p(D)\bigr) = \{0\}$. Because the kernel of the restricted projection $\left.\pi\right|_F$ is exactly this intersection, $\left.\pi\right|_F : F \to V_q$ is an injective linear map.

Because $\left.\pi\right|_F$ injects $F$ into $V_q$, we must have $\dim F \le \dim V_q$. By Proposition~\ref{prop:annihilator-dimension}, $\dim V_q = \deg q(D)$. 

Combining our inequalities yields $\deg q(D) \le \dim F \le \deg q(D)$, forcing $\dim F = \dim V_q = \deg q(D)$. Since $\left.\pi\right|_F$ is an injective linear map between finite-dimensional vector spaces of the same dimension, it is necessarily a bijection (an isomorphism).

Consequently, the trial space $F$ acts as an internal complement (a transversal) to $\ker\bigl(p(D)\bigr)$ within $\ker\bigl(p(D)q(D)\bigr)$. Searching for the particular solution by restricting the operator to the trial space $F$ (the Method of Undetermined Coefficients) is algebraically equivalent to selecting the unique affine coset in the quotient space $V_q$ (the Annihilator Method).
\end{proof}

\begin{Remark}
This isomorphism elegantly grounds the abstract algebra in concrete dimensions. For the forcing function $f(x) = x^{k_0} e^{\mu x}$, the Method of Undetermined Coefficients constructs the trial space $F_{m_0}$, which has dimension $k_0+1$. Concurrently, the Annihilator Method identifies the minimal annihilator $q(D) = (D-\mu)^{k_0+1}$, producing a quotient space $V_q$ of dimension $\deg q(D) = k_0+1$. The theorem guarantees that searching for coefficients within the $(k_0+1)$-dimensional vector space $F_{m_0}$ is mathematically indistinguishable from isolating the unique affine coset within the $(k_0+1)$-dimensional quotient space $V_q$.
\end{Remark}

To conclude this section, we observe that the minimal annihilators encode the procedural rules of classical solution methodologies.

\begin{Lemma}[Minimal Annihilators of Standard Functions]\label{lem:min-annihilators}
The following algebraic statements hold:
\begin{enumerate}
\item If $f(x)$ is a polynomial of degree $n$, a minimal annihilator is $q(D)=D^{n+1}$.
\item If $f(x)=p(x)e^{\lambda x}$, where $p(x)$ has degree $n$, a minimal annihilator is $q(D)=(D-\lambda)^{n+1}$.
\item If $f(x)=\sin(\alpha x)$ or $f(x)=\cos(\alpha x)$, a minimal annihilator is $q(D)=D^2+\alpha^2$.
\item If $f(x)=p(x)e^{\lambda x}\sin(\alpha x)$ or $f(x)=p(x)e^{\lambda x}\cos(\alpha x)$, where $p(x)$ has degree $n$, a minimal annihilator is $q(D)=\bigl((D-\lambda)^2+\alpha^2\bigr)^{n+1}$.
\item If $p(D)$ minimally annihilates $f$ and $q(D)$ minimally annihilates $g$, then a minimal annihilator for the sum $f+g$ is the least common multiple:
\[
\operatorname{lcm}\bigl(p(D),q(D)\bigr).
\]
\end{enumerate}
\end{Lemma}

\begin{proof}
Statements (1) through (4) follow iteratively from repeated differentiation and standard identities such as $(D-\lambda)e^{\lambda x}=0$ and $(D^2+\alpha^2)\sin(\alpha x)=0$.

For statement (5), if an operator $r(D)$ annihilates $f+g$, it must annihilate both individually, and hence must constitute a common multiple of $p(D)$ and $q(D)$. Because $\mathbb{C}[D]$ is a principal ideal domain, the common multiples form the ideal generated by $\operatorname{lcm}(p(D),q(D))$, which uniquely determines the minimal degree.
\end{proof}

\begin{Remark}
Lemma~\ref{lem:min-annihilators} explains the classical ``guessing rules'' for undetermined coefficients. Each additional factor in the operator $\operatorname{lcm}(p(D),q(D))$ mathematically necessitates an increase in the dimension of the required trial space, which physically manifests as the multiplication by powers of $x$ to resolve resonance.
\end{Remark}

\section{Extension to Variable Coefficients: Ring Isomorphisms}

The algebraic framework articulated thus far appears constrained to constant coefficients. However, by substituting the generating operator, one can resolve broad classes of variable-coefficient equations.

\subsection{The Euler Operator}
Using the standard dilation operator $D=\frac{d}{dx}$, a Cauchy-Euler equation can be expressed in terms of the operator $\theta = xD$. 
The fundamental identity is $(xD)^2y = xD(xy') = xy' + x^2y''$. Therefore, the standard second-order Euler equation $ax^2y''+bxy'+cy=0$ can be formulated as
\[
\bigl(a\theta^2+(b-a)\theta+c\bigr)y=0.
\]
To determine the eigenfunctions of $\theta$, we solve the equation $\theta(f(x)) = m f(x)$, or equivalently, $x f'(x) = m f(x)$. Up to a constant multiple, the unique solution is $f(x) = x^m$. Thus, $x^m$ acts as the universal eigenfunction of $\theta$ with eigenvalue $m$. Consequently, $p(\theta)x^m=p(m)x^m$, and the differential equation reduces to the algebraic roots of $p(m)=0$.

\begin{Example}[Algebraic Reduction of an Euler Equation]
Consider the second-order variable-coefficient equation:
\[
x^2 y'' - 3xy' + 3y = 0.
\]
Utilizing the operator $\theta = xD$, we substitute $x^2 y'' = \theta(\theta-1)y$ and $xy' = \theta y$ to obtain:
\[
(\theta^2 - \theta - 3\theta + 3)y = (\theta^2 - 4\theta + 3)y = (\theta-1)(\theta-3)y = 0.
\]
The algebraic roots are $m=1$ and $m=3$. Because the basis eigenfunctions of the operator are $x^m$, the general solution is deduced as $y = c_1 x + c_2 x^3$.
\end{Example}

\subsection{Generalized Operators and $\mathbb{C}[D_t]$}
The algebraic success of the Euler operator $\theta = x \frac{d}{dx}$ is not a coincidence; it is the manifestation of a strict ring isomorphism. The substitution $t = \ln x$ transforms the variable-coefficient operator $\theta$ exactly into the constant-coefficient derivative $D_t = \frac{d}{dt}$, which maps the continuous eigenfunctions perfectly as $x^m = e^{mt}$.

This mechanism generalizes to a broader class of operators. Let $\Theta = f(x)\frac{d}{dx}$. We seek a change of variables $t = g(x)$ that transforms $\Theta$ into the standard constant-coefficient derivative $D_t = \frac{d}{dt}$. The chain rule dictates:
\[
\Theta = f(x)\frac{dt}{dx}\frac{d}{dt}.
\]
For $\Theta = \frac{d}{dt}$, we must have $f(x)\frac{dt}{dx} = 1$, which yields the integral substitution:
\[
t = \int \frac{1}{f(x)}\,dx.
\]

\begin{figure}[htbp]
\centering
\begin{tikzpicture}[node distance=3.5cm, auto, >=stealth, thick]
    \node (A) {$C^\infty(\mathbb{R}_x)$};
    \node (B) [right of=A, node distance=4.5cm] {$C^\infty(\mathbb{R}_x)$};
    \node (C) [below of=A] {$C^\infty(\mathbb{R}_t)$};
    \node (D) [right of=C, node distance=4.5cm] {$C^\infty(\mathbb{R}_t)$};

    \draw[->] (A) to node {$p(\Theta)$} (B);
    \draw[->] (C) to node [swap] {$p(D_t)$} (D);

    \draw[->, blue] (A) to node [swap] {$t = g(x)$} (C);
    \draw[->, blue] (B) to node {$t = g(x)$} (D);
    
    \node[align=center, blue, font=\small] at (2.25, -0.6) {Ring Isomorphism \\ $\Theta \mapsto D_t$};
\end{tikzpicture}
\caption{The commutative diagram illustrating the structural ring isomorphism. The variable-coefficient operator $p(\Theta)$ acting on functions of $x$ is isomorphic to the constant-coefficient operator $p(D_t)$ acting on functions of $t$ via the substitution $t = g(x)$.}
\label{fig:isomorphism}
\end{figure}

This substantiates a structural theorem: Any variable-coefficient differential equation that can be formulated as a polynomial $p(\Theta)y = 0$ is isomorphic to a constant-coefficient equation $p(D_t)y = 0$. For instance, selecting $f(x) = x$ yields the standard Euler substitution $t = \ln x$, while selecting $f(x) = ax+b$ generates Legendre's linear equation.

\subsection{Reduction of Order and Non-Commutative Factorization}
A ubiquitous technique in classical differential equations is d'Alembert's ``reduction of order.'' Given a second-order variable-coefficient operator $L = D^2 + p(x)D + q(x)$ and a known basis solution $y_1 \in \ker L$, students are traditionally instructed to seek a second solution of the form $y_2 = v(x)y_1$. Algebraically, this equates to searching for the remainder of the kernel within the ideal $(y_1)$ generated by the first solution. 

This substitution is not an analytic trick; it is the manifestation of algebraic factorization within the non-commutative ring of variable-coefficient operators, $\mathbb{C}(x)[D]$. 

Despite the non-commutativity of this ring ($D \circ x \neq x \circ D$), identifying a solution $y_1$ is algebraically equivalent to identifying a linear right-factor of the operator. Specifically, the first-order operator that minimally annihilates $y_1$ is precisely $(D - \frac{y_1'}{y_1})$. Because $y_1 \in \ker L$, the full second-order operator must factor as:
\[
L = \left(D + p(x) + \frac{y_1'}{y_1}\right) \circ \left(D - \frac{y_1'}{y_1}\right).
\]
By restricting our search to the ideal $y \in (y_1)$, we set $y = v y_1$. The inner operator acts on this ideal via the product rule, yielding:
\[
\left(D - \frac{y_1'}{y_1}\right)(v y_1) = v'y_1 + v y_1' - v y_1' = v'y_1.
\]
The inner operator annihilates the zeroth-order $v$ term, leaving only the derivative $v'$. Consequently, the full equation $L[vy_1] = 0$ collapses into a first-order equation in the variable $w = v'$. 

\begin{Remark}[Differential Galois Theory]
This algebraic reduction is the differential analogue of classical Galois theory. Just as finding a root $r_1$ of a polynomial allows one to divide out the factor $(x - r_1)$ to lower the polynomial's degree, finding a solution $y_1$ allows one to factor out the operator $(D - \frac{y_1'}{y_1})$ to lower the differential equation's order. This structural parallel forms the foundational entry point into Picard-Vessiot theory \cite{VanDerPutSinger}.
\end{Remark}

\section{The Boundary of Linear Theory: Diffeomorphic Linearization}

The algebraic framework developed in this article relies on the premise that the differential equations are linear, ensuring that the solution spaces form valid $\mathbb{C}[D]$-modules. When a differential equation is non-linear, superposition fails, the solution space loses its vector space structure, and the differential operator can no longer be factored within a polynomial ring. 

However, there exists an algebraic mechanism for rescuing specific non-linear equations by mapping them back into the linear framework. We employ a process known as \emph{diffeomorphic linearization}. This involves applying a smooth, invertible coordinate transformation (a diffeomorphism) $\Phi$ to the dependent variable, thereby mapping a curved, non-linear solution manifold\footnote{Formally, $C^\infty(\mathbb{R})$ constitutes an infinite-dimensional Fréchet space, and the solution sets form submanifolds within it. For the scope of this algebraic framework, we employ the term \emph{manifold} in a descriptive geometric sense to contrast these curved solution sets with flat linear vector modules.} into an image space that inherits the structure of a linear vector space, allowing it to function as a valid $\mathbb{C}[D]$-module.

\subsection{The Bernoulli Equation and Space Flattening}
A historical example of this technique is the Bernoulli differential equation. To illustrate its connection to our established ring theory, consider the constant-coefficient case:
\[
\frac{dy}{dx} + ay = by^r, \qquad r \neq 0, 1.
\]
Because of the $y^r$ term, the operator $L[y] = y' + ay - by^r$ is not a module endomorphism. The solution space is not a linear kernel, but rather a non-linear manifold $Y \subset C^\infty(\mathbb{R})$. 

To resolve this algebraically, we seek an isomorphism of the function space itself that flattens this non-linear manifold back into a linear vector space. Consider the algebraic transformation of the dependent variable:
\[
\Phi : y \mapsto u = y^{1-r}.
\]
Applying the chain rule, $\frac{du}{dx} = (1-r)y^{-r}\frac{dy}{dx}$. Multiplying the original Bernoulli equation by $(1-r)y^{-r}$ transforms the equation into:
\[
\frac{du}{dx} + (1-r)au = (1-r)b.
\]
This resulting equation is linear. More importantly, because we restricted our initial equation to constant coefficients, the diffeomorphism $\Phi$ has mapped the problem directly back into our foundational polynomial ring. The image space $U = \Phi(Y)$ now functions as our $\mathbb{C}[D]$-module, governed by the linear operator:
\[
L_u = D + (1-r)a \in \mathbb{C}[D].
\]
We can now apply the exact algebraic machinery developed in Section 4. The forcing term is the constant $f(x) = (1-r)b$. By Lemma~\ref{lem:min-annihilators}, the minimal annihilator for a constant is $D$. Applying this annihilator yields the homogeneous operator $D\bigl(D + (1-r)a\bigr)u = 0$. The roots of the characteristic polynomial dictate that the solution must reside in the space $\operatorname{span}_{\mathbb{C}}\{1, e^{-(1-r)ax}\}$. 

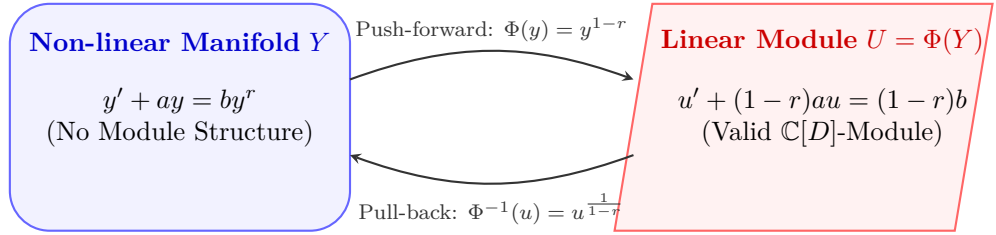
\begin{figure}[htbp]
\centering
\begin{tikzpicture}[node distance=4cm, auto, >=stealth, thick]
    \filldraw[fill=blue!5, draw=blue!60, rounded corners=15pt] (-5,-1.5) rectangle (-0.5,1.5);
    \node[blue!80!black] at (-2.75, 1) {\textbf{Non-linear Manifold} $Y$};
    \node[align=center] at (-2.75, 0) {$y' + ay = by^r$ \\ (No Module Structure)};

    \filldraw[fill=red!5, draw=red!60] (3,-1.5) -- (7.5,-1.5) -- (8,1.5) -- (3.5,1.5) -- cycle;
    \node[red!80!black] at (5.75, 1) {\textbf{Linear Module} $U = \Phi(Y)$};
    \node[align=center] at (5.75, 0) {$u' + (1-r)au = (1-r)b$ \\ (Valid $\mathbb{C}[D]$-Module)};

    \draw[->, color=black!80] (-0.5, 0.5) to[bend left=20] node[above] {\footnotesize{Push-forward: $\Phi(y) = y^{1-r}$}} (3.25, 0.5);
    \draw[->, color=black!80] (3.25, -0.5) to[bend left=20] node[below] {\footnotesize{Pull-back: $\Phi^{-1}(u) = u^{\frac{1}{1-r}}$}} (-0.5, -0.5);
\end{tikzpicture}
\caption{The geometry of diffeomorphic linearization. The non-linear transformation $\Phi$ pushes the curved solution manifold $Y$ into the ambient function space, creating an image space $U = \Phi(Y)$ that acts as a valid linear $\mathbb{C}[D]$-module. The equation is solved entirely within this linear module, and the resulting affine solution coset is pulled back to the original manifold via the inverse mapping $\Phi^{-1}$.}
\label{fig:push_pull}
\end{figure}

Once the affine coset of solutions is identified within this linear $\mathbb{C}[D]$-module, the solutions are pulled back to the original non-linear manifold via the inverse mapping $\Phi^{-1}(u) = u^{\frac{1}{1-r}}$, allowing us to solve the non-linear equation via finite-dimensional linear algebra.

\begin{Remark}[Diffeomorphisms vs. Ring Isomorphisms]
It is conceptually crucial to distinguish this transformation from the Euler substitution discussed in Section 5. The Euler substitution $t = \ln x$ is a change of the \emph{independent} variable, which induces a \textbf{ring isomorphism} on the differential operators ($\Theta \mapsto D_t$), mapping a variable-coefficient ring to a constant-coefficient ring. Conversely, the Bernoulli substitution $u = y^{1-r}$ is a change of the \emph{dependent} variable. It acts as a \textbf{diffeomorphism} on the function space itself, mapping a non-linear solution manifold directly into an image space $\Phi(Y)$ that acts as a valid $\mathbb{C}[D]$-module, where standard operator algebra can execute the solution.
\end{Remark}

\subsection{The Riccati Equation and Dimensional Elevation}
A second notable example of diffeomorphic linearization is the Riccati equation. While the Bernoulli substitution preserves the order of the differential equation, the Riccati transformation absorbs a quadratic non-linearity by algebraically elevating the dimension of the target linear module. 

Consider the constant-coefficient Riccati equation, characterized by its quadratic non-linearity:
\[
\frac{dy}{dx} = a + by + cy^2, \qquad c \neq 0.
\]
Because of the $y^2$ term, this equation resides on a non-linear manifold $Y$ outside our linear module framework. To linearize this space, we can bootstrap a substitution directly from the structure of the non-linearity. Isolating the purely non-linear core of the equation, $y' \approx cy^2$, algebraic rearrangement yields $y \approx \frac{y'}{cy}$. This self-referential form naturally suggests decoupling the numerator and denominator by introducing a new auxiliary function $u$, motivating a substitution of the form $y \propto \frac{u'}{cu}$. 

The quotient rule provides the mechanical confirmation of this intuition: evaluating the derivative of the base guess $y = \frac{u'}{cu}$ yields $y' = \frac{u''}{cu} - c\left(\frac{u'}{cu}\right)^2$, which simplifies exactly to $y' = \frac{u''}{cu} - cy^2$. To ensure this negative squared term cancels the positive $+cy^2$ non-linearity in the original equation, we negate the substitution. This defines a diffeomorphism $\Phi$, formally expressed by $u = \exp\bigl(-c \int y \, dx\bigr)$, but most practically implemented via its inverse substitution:
\[
\Phi(y) = u \implies y = -\frac{1}{c} \frac{u'}{u}.
\]
Applying the quotient rule to differentiate this substitution yields:
\[
\frac{dy}{dx} = -\frac{1}{c} \left( \frac{u''u - (u')^2}{u^2} \right) = -\frac{u''}{cu} + \frac{(u')^2}{cu^2}.
\]
Substituting both $y$ and $y'$ back into the original Riccati equation, we obtain:
\[
-\frac{u''}{cu} + \frac{(u')^2}{cu^2} = a - \frac{b}{c}\frac{u'}{u} + c\left(-\frac{u'}{cu}\right)^2.
\]
The non-linear $(u')^2$ term generated by the quotient rule cancels the quadratic non-linearity on the right-hand side. Multiplying the entire equation by $-cu$ linearizes the relationship, yielding:
\[
u'' - bu' + acu = 0.
\]
This algebraic transformation is instructive: it has mapped a first-order non-linear manifold $Y$ into an image space $U = \Phi(Y)$ that operates as a second-order linear vector space. Once again, by restricting our analysis to constant coefficients, the target module is placed back into our foundational polynomial ring. The governing operator acting on $\Phi(Y)$ is now:
\[
L_u = D^2 - bD + ac \in \mathbb{C}[D].
\]
We can readily construct the basis for this zero-eigenspace by finding the roots of the characteristic polynomial $\lambda^2 - b\lambda + ac = 0$. Once the linear basis for $u(x)$ is established within the $\mathbb{C}[D]$-module, the solution is pulled back to the original non-linear manifold via the inverse relation $y(x) = -\frac{1}{c} \frac{u'(x)}{u(x)}$. 

Together, the Bernoulli and Riccati equations establish a boundary principle for our algebraic framework. When classical operator theory fails due to non-linearity, it can frequently be restored by a diffeomorphic linearization that maps the curved geometry into an image space $\Phi(Y)$, either flattening the space (Bernoulli) or elevating its linear dimension (Riccati) to inherit the standard structure of a $\mathbb{C}[D]$-module.

\section{The Algebraic Breakdown and Integral Inverses}

\subsection{Torsion-Free Forcing and the Limits of Algebra}
It is crucial to recognize that the Method of Undetermined Coefficients and the Annihilator Method are not universally applicable to all continuous functions. They are bound to the algebraic constraints of the torsion submodule $M_{\text{tors}}$. Both methods fail if the forcing function $f(x)$ is \emph{torsion-free}.

For the Annihilator Method, the failure is definitional: if $f(x) \notin M_{\text{tors}}$, then by definition, no polynomial operator $q(D) \in \mathbb{C}[D]$ exists such that $q(D)f = 0$. Consequently, one cannot construct the homogeneous equation $q(D)p(D)y = 0$, and the method collapses.

For the Method of Undetermined Coefficients, the failure is structural. The foundational requirement of a trial space $F$ is that it must be invariant under the derivative operator $D$. Therefore, the trial space must span $f(x)$ and all of its linearly independent derivatives. If $f(x) \in M_{\text{tors}}$ (a quasipolynomial), the sequence of derivatives will eventually loop or terminate, generating a finite-dimensional trial space. This allows the restriction $\left.p(D)\right|_F$ to be expressed as a finite matrix, rendering the coefficients solvable via standard linear algebra. 

However, if $f(x)$ is torsion-free (e.g., $\ln(x)$ or $\tan(x)$), repeated differentiation generates an infinite sequence of linearly independent functions. The required trial space becomes infinite-dimensional. Attempting to balance coefficients transforms from a finite matrix inversion into an infinite-dimensional algebraic system.

\begin{figure}[htbp]
\centering
\begin{tikzpicture}[scale=1]
    \draw[thick, fill=gray!5] (-4,-3) rectangle (4,3);
    \node[anchor=north west, font=\bfseries] at (-4,3) {$C^\infty(\mathbb{R})$};

    \filldraw[thick, fill=blue!10, draw=blue!80] (0, 0.3) ellipse (3cm and 1.8cm);
    \node[blue!80!black, font=\bfseries] at (0, 1.35) {Torsion Submodule $M_{\text{tors}}$};
    \node[align=center, font=\footnotesize] at (0, 0.2) {
        $x^k e^{\lambda x}$, $\sin(\alpha x)$, polynomials\\[0.15cm]
        \textbf{Algebraic Methods Succeed:}\\
        Finite-dimensional trial spaces
    };

    \node[align=center, font=\footnotesize] at (0, -2.1) {
        \textbf{Torsion-Free (Algebra Fails):}\\
        Infinite-dimensional trial spaces. Must return to Analysis\\
        (Variation of Parameters / Integration).
    };
\end{tikzpicture}
\caption{The algebraic classification of $C^\infty(\mathbb{R})$ relative to the ring $\mathbb{C}[D]$. The Method of Undetermined Coefficients and the Annihilator Method are operational only when the forcing function resides within the torsion submodule $M_{\text{tors}}$. If the function escapes this finite-dimensional operator algebra, the structural constraints fail, forcing a return to continuous analysis.}
\label{fig:torsion_venn}
\end{figure}
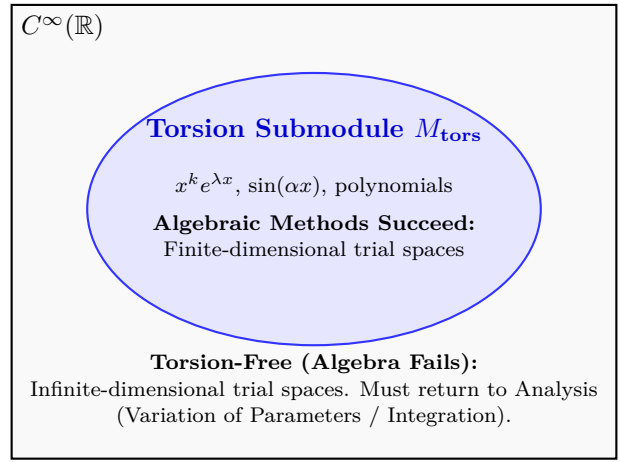

Ultimately, when a forcing function lies outside the torsion submodule, it escapes the finite-dimensional operator algebra of the polynomial ring $\mathbb{C}[D]$. In these instances, the discrete algebraic framework must be extended to continuous analysis---specifically, explicit integral inverses---to resolve the non-homogeneous equation.

\subsection{Mikusi\'nski's Calculus and Green's Functions}
Mikusi\'nski's operational calculus \cite{Mikusinski} formalizes this operator inversion by constructing a field of fractions $\mathcal{F} = \operatorname{Frac}(\mathbb{C}[D])$. In this operational field, the formal solution to $p(D)y = f$ is the algebraic quotient $y = \frac{1}{p(D)} \cdot f$.

To map this abstract fraction back into the domain of smooth functions, we recognize that multiplication in $\mathcal{F}$ corresponds to the analytic operation of \emph{convolution} ($*$), and the multiplicative identity $1$ corresponds to the Dirac delta distribution $\delta$. Finding the algebraic inverse $\frac{1}{p(D)}$ equates to identifying an element $G$ such that $p(D) * G = \delta$. This $G$ constitutes the \emph{Green's function} (or fundamental solution) of the corresponding operator. 

\subsection{Constructing Integral Transforms: Kernels and Eigenfunctions}
While convolution supplies a time-domain inverse, integral transforms provide a direct mapping into an algebraic frequency domain. Every linear operator that maps a function domain ($x$) to a frequency domain ($s$) takes the form of an integral transform:
\[
T\{f(x)\}(s) = \int_a^b K(x, s) f(x) \,dx.
\]
The fundamental mechanism for constructing a transform that diagonalizes a specific differential operator $L$ is to define the integral kernel $K(x, s)$ as the inverse of the eigenfunctions of $L$. 

The universal eigenfunctions of the constant-coefficient derivative $D = \frac{d}{dt}$ are $e^{st}$. Therefore, to map this operator to pure algebra, we construct a transform utilizing the inverse eigenfunction $K(t,s) = e^{-st}$. This procedure yields the \emph{Laplace Transform}. 

Applying the Laplace transform maps the operator $D$ to scalar multiplication by $s$, translating the differential equation $p(D)y = f(t)$ into $p(s)Y(s) = F(s)$. The formal solution is $Y(s) = \frac{F(s)}{p(s)}$.
The analytic inverse transform is executed via the algebraic algorithm of partial fraction decomposition, which serves as the algebraic analogue of decomposing $\ker p(D)$ into direct sums of exponential and generalized exponential eigenspaces.

\begin{Remark}[Operator Rationalization and Spectral Factorization]
By restricting the Laplace transform to the imaginary axis ($s = i\omega$), we isolate the system's response to pure oscillatory eigenfunctions, $e^{i\omega t}$. Under this mapping, the second derivative operator $D^2$ maps to the real scalar $(i\omega)^2 = -\omega^2$. Consequently, any differential operator belonging strictly to the subring $\mathbb{C}[D^2]$ maps to a purely real polynomial in $-\omega^2$. Physically, this means operators in $\mathbb{C}[D^2]$ scale amplitude but induce zero complex phase shift.

This algebraic property explains the classical ODE heuristic of ``multiplying by the conjugate'' to solve equations with trigonometric forcing, such as $p(D)y = \sin(\omega t)$. By applying the conjugate operator $p(-D)$ to both sides, the governing operator becomes $p(-D)p(D)$. Because all odd powers of $D$ cancel, this product is forced down into the real subring $\mathbb{C}[D^2]$. In the frequency domain, this structural projection evaluates to $p(i\omega)p(-i\omega) = |p(i\omega)|^2$. 

Thus, the algebraic act of rationalizing an operator to eliminate odd derivatives is equivalent to the continuous operation of \emph{spectral factorization}---stripping away the system's phase shift to isolate its pure squared-magnitude power spectrum.
\end{Remark}

This identical theoretical framework applies to the variable-coefficient Euler ring. The eigenfunctions of $\theta = x \frac{d}{dx}$ are the monomials $x^s$. To map Euler equations to algebraic rational functions, we formulate a transform over the domain $(0, \infty)$. Because the natural translation-invariant measure for this multiplicative domain is $\frac{dx}{x}$, integrating against the eigenfunction $x^s$ produces the kernel $K(x, s) = x^s \frac{1}{x} = x^{s-1}$. 

This constructed operator defines the standard \emph{Mellin Transform}:
\[
\mathcal{M}\{f\}(s) = \int_0^\infty f(x) x^{s-1} \,dx.
\]
Evaluated via integration by parts, this transform maps the Euler operator $xD_x$ to simple multiplication by $-s$. This demonstrates that the Laplace transform is not an ad hoc mechanism, but rather one instance of a broader algebraic principle: integral transforms are specific kernels custom-tailored to be the algebraic inverses of their respective differential rings.

\section{The Multivariable Generalization: PDEs and the Limits of the PID}

There exists a parallel between the theory of linear ordinary differential equations (ODEs) with constant coefficients and the theory of linear partial differential equations (PDEs) with constant coefficients. By expanding our algebraic viewpoint from $\mathbb{C}[D]$ to the multivariable ring $\mathbb{C}[D_x, D_y]$, we can push the operator-theoretic framework to its natural limits.

\subsection{Homogeneous Equations and Characteristic Varieties}

Consider the homogeneous ODE $p(D)y=0$, where $D=\frac{d}{dx}$ and $p$ is a polynomial.
The conventional approach seeks solutions of the form $y=e^{\lambda x}$.
Since $D e^{\lambda x}=\lambda e^{\lambda x}$, it follows that $p(D)e^{\lambda x}=p(\lambda)e^{\lambda x}$.
Therefore, $e^{\lambda x}$ is a solution if and only if $p(\lambda)=0$. This generates the familiar characteristic equation.

Now consider the homogeneous PDE $p(D_x,D_y)u=0$, where $D_x=\frac{\partial}{\partial x}$ and $D_y=\frac{\partial}{\partial y}$.
Seeking solutions of the form $u=e^{\lambda x+\mu y}$, we observe:
\[
D_x e^{\lambda x+\mu y}=\lambda e^{\lambda x+\mu y},
\qquad
D_y e^{\lambda x+\mu y}=\mu e^{\lambda x+\mu y}.
\]
Consequently, $p(D_x,D_y)e^{\lambda x+\mu y} = p(\lambda,\mu)e^{\lambda x+\mu y}$.
Hence, $e^{\lambda x+\mu y}$ constitutes a solution precisely when $p(\lambda,\mu)=0$.
Thus, the discrete characteristic equation of an ODE becomes a continuous \emph{characteristic variety} $p(\lambda,\mu)=0$ for a PDE.

\subsection{Canonical Examples in Mathematical Physics}

To demonstrate the scope of this algebraic framework, we observe that the foundational equations of mathematical physics are specific polynomials residing within the ring $\mathbb{C}[D_x, D_y, D_t]$. By translating these operators into our algebraic notation, their geometric characteristic varieties are exposed.

\begin{itemize}
\item \textbf{The Heat Equation:} The equation $u_t = \alpha u_{xx}$ maps to the operator $p(D_t, D_x) = D_t - \alpha D_x^2$. 
    Its characteristic variety is the parabola $\lambda - \alpha \mu^2 = 0$. The basis eigenfunctions are $e^{\alpha \mu^2 t + \mu x}$.
\item \textbf{The Wave Equation:} The equation $u_{tt} = c^2 u_{xx}$ maps to the operator $p(D_t, D_x) = D_t^2 - c^2 D_x^2$. 
    Its characteristic variety factors as $(\lambda - c\mu)(\lambda + c\mu) = 0$, representing two intersecting lines corresponding to forward and backward traveling waves.
\item \textbf{Laplace's Equation:} The two-dimensional equation $\Delta u = 0$ maps to $p(D_x, D_y) = D_x^2 + D_y^2$. 
    Its characteristic variety is $\lambda^2 + \mu^2 = 0$. Over the complex field, this factors into $\lambda = \pm i\mu$, forcing the eigenfunctions $e^{\pm i\mu x + \mu y}$ which linearly combine to form the classic harmonic functions.
\item \textbf{The Klein-Gordon Equation:} The equation $u_{tt} - c^2 u_{xx} + m^2 u = 0$ maps to the operator $p(D_t, D_x) = D_t^2 - c^2 D_x^2 + m^2$. 
    Its characteristic variety is the hyperbola $\lambda^2 - c^2 \mu^2 + m^2 = 0$, reflecting the dispersive nature of the associated quantum waves.
\end{itemize}

In each case, identifying the algebraic variety dictates the allowable parameters for the plane wave solutions $e^{\lambda t + \mu x}$, bypassing the need for early analytic trial-and-error.

\subsection{Factorization of the Differential Operator}

For ODEs, factorization of the characteristic polynomial leads directly to the solution structure.
For example, $(D-\lambda_1)(D-\lambda_2)y=0$ yields the general solution $y=c_1e^{\lambda_1 x}+c_2e^{\lambda_2 x}$.

Similarly, for PDEs such as $(D_x^2-D_y^2)u=0$, we factor the operator:
\[
D_x^2-D_y^2 = (D_x-D_y)(D_x+D_y).
\]
The resultant first-order equations $(D_x-D_y)u=0$ and $(D_x+D_y)u=0$ admit the solutions $u=F(x+y)$ and $u=G(x-y)$, respectively.
Therefore, the general solution is $u=F(x+y)+G(x-y)$.
Operator factorization clearly plays an analogous structural role in both theoretical frameworks.

\subsection{Non-Homogeneous Equations and Resonance}

The method of undetermined coefficients extends naturally to constant-coefficient PDEs.
Consider $p(D_x,D_y)u=f(x,y)$.
If $f(x,y)=e^{\lambda x+\mu y}$, we assert the ansatz $u_p=Ae^{\lambda x+\mu y}$.
Since $p(D_x,D_y)e^{\lambda x+\mu y} = p(\lambda,\mu)e^{\lambda x+\mu y}$, we obtain $A=\frac{1}{p(\lambda,\mu)}$ whenever $p(\lambda,\mu)\neq0$.

Furthermore, resonance behaves similarly. For ODEs, if $p(\lambda)=0$, the trial solution must be multiplied by a power of $x$. For example, $(D-1)y=e^x$ necessitates the ansatz $y_p=Axe^x$.
Likewise, for PDEs, if $p(\lambda,\mu)=0$, the naive ansatz fails. For example, $(D_x^2-D_y^2)u=e^{x+y}$ is resonant because $1^2-1^2=0$. A successful ansatz is obtained by multiplying $e^{x+y}$ by an appropriate linear polynomial in $x$ and $y$.

\subsection{Spectral Interpretation and the Algebraic Rupture}

The unifying principle underlying both theories is that exponentials are eigenfunctions of constant-coefficient differential operators.
Solving constant-coefficient ODEs and PDEs begins as fundamentally the identical algebraic process: differential operators map to polynomials, and exponentials serve as eigenfunctions. 

However, expanding the mathematical scope to $\mathbb{C}[D_x, D_y]$ transforms this comparison into an important lesson regarding the boundaries of algebra: \emph{the transition from one variable to multiple variables destroys the Principal Ideal Domain (PID) property.}

In our treatment of ODEs, every major theorem relied on the fact that $\mathbb{C}[D]$ is a PID. Upon transitioning to $\mathbb{C}[D_x, D_y]$, Hilbert's Basis Theorem guarantees the ring remains Noetherian, but it ceases to be a PID \cite{Coutinho}. This induces significant consequences for the algebraic frameworks utilized in ODEs:
\begin{itemize}
\item \textbf{Infinite-Dimensional Kernels:} In $\mathbb{C}[D]$, the kernel of the simplest operator $D$ is a finite-dimensional space spanned by a constant. In $\mathbb{C}[D_x, D_y]$, the kernel of $D_x$ expands to the infinite-dimensional space of all arbitrary smooth functions $f(y)$. The eigenspaces are no longer finitely generated by simple exponentials, explaining why the PDE solutions detailed in Section 8.3 incorporate arbitrary functions $F(x+y)$ rather than discrete constants.
\item \textbf{Non-Principal Ideals and Compatibility Conditions:} In $\mathbb{C}[D]$, every ideal is generated by a single polynomial. Consequently, the equation $Dy = f(x)$ locally admits a solution via direct integration. In contrast, consider the ideal generated by two operators in $\mathbb{C}[D_x, D_y]$, namely $I = \langle D_x, D_y \rangle$. Attempting to solve the system $D_x u = f(x,y)$ and $D_y u = g(x,y)$ reveals that a solution exists if and only if the forcing functions satisfy a compatibility condition governed by Clairaut's theorem: $D_y f = D_x g$. The emergence of compatibility conditions is the direct analytic consequence of working with a non-principal ideal.
\item \textbf{The Loss of Coprimality:} In Section 3, we leveraged Bézout's identity for coprime polynomials to split ODE solution spaces into direct sums. Because ideals in $\mathbb{C}[D_x, D_y]$ may be generated by multiple polynomials, coprime ideals do not unconditionally split the function space, causing the annihilator method to fragment.
\end{itemize}

\begin{Example}[Compatibility Conditions in PDEs]
To demonstrate the consequence of non-principal ideals explicitly, consider attempting to find a function $u(x,y)$ such that:
\[
D_x u = y^2 \quad \text{and} \quad D_y u = x.
\]
Integrating the first equation yields $u(x,y) = xy^2 + h(y)$. Differentiating this result with respect to $y$ requires $D_y u = 2xy + h'(y)$. Equating this to the second equation mandates $2xy + h'(y) = x$, which is impossible for any function $h$ depending strictly on $y$. The system has no solution because the required compatibility condition, $D_y(y^2) = D_x(x)$, fails ($2y \neq 1$). In the PID $\mathbb{C}[D]$, such contradictions cannot exist for single-variable integration.
\end{Example}

Analyzing this algebraic rupture provides a pedagogical bridge. PDE systems become more difficult not simply because the multivariable calculus is computationally heavier, but because the underlying ring structure has fragmented. Indeed, this represents the historical entry point into $D$-module theory: because $\mathbb{C}[D_x, D_y]$ is not a PID, mathematicians were required to invent $D$-modules and sheaf theory to manage the complex compatibility conditions of multivariable operator ideals \cite{Coutinho}.

\section{Conclusion}

By analyzing linear differential equations through the lens of polynomial rings and module theory, classical methods are seen as more than a collection of disconnected computational heuristics. The annihilator method emerges naturally from quotient modules, variable-coefficient substitutions are mathematically verified as ring isomorphisms, reduction of order is revealed as non-commutative operator factorization, and integral transforms are recognized as custom-built algebraic inverses constructed from eigenfunctions. Furthermore, extending our view to the non-linear equations of Bernoulli and Riccati illustrates that when linear operator theory fails, diffeomorphic mapping can often recover it. Finally, contrasting ODEs with the multivariable ring of PDEs illustrates both the operational power and the structural limitations of this algebraic framework. This perspective unifies the underlying mathematics and elevates its pedagogical presentation.

Beyond the scope of continuous, constant-coefficient equations, these structural principles are universal. For instance, the framework developed here applies identically to discrete linear recurrence relations (difference equations) by replacing the derivative $D$ with the forward shift operator $E$. Because the resulting polynomial ring $\mathbb{C}[E]$ is also a principal ideal domain, the theories of characteristic roots, polynomial multipliers for resonance, and the annihilator method map over isomorphically \cite{KelleyPeterson}. Moreover, the algebraic reduction of the Euler equation serves as the foundational model for analyzing regular singular points via the Frobenius method, formally necessitating the use of the non-commutative Weyl algebra $\mathbb{C}[x, D]$ \cite{Coutinho}. Ultimately, recognizing these algebraic symmetries allows differential equations to be taught not merely as a collection of analytical techniques, but as a connected structural discipline.

\end{document}